\documentclass[11pt,reqno]{amsart}

\usepackage{amssymb,amsmath,amsthm,amscd,latexsym,amsfonts}
\usepackage{mathtools}
\usepackage[T1]{fontenc}

\usepackage{graphicx}
\usepackage{xcolor}
\usepackage{cite}

\usepackage[a4paper,top=3cm, bottom=3cm, left=3cm, right=2cm]{geometry}

\newtheorem{thm}{Theorem}
\newtheorem{defn}{Definition}
\newtheorem{lemma}{Lemma}
\newtheorem{pro}{Proposition}
\newtheorem{rk}{Remark}

\newtheorem{ex}{Example}

\numberwithin{equation}{section} 

\newcommand{\M}{{\mathcal M}}

\newcommand{\bea}{\begin{eqnarray}}
\newcommand{\eea}{\end{eqnarray}}

\newcommand{\Q}{\mathbb{Q}}
\newcommand{\R}{\mathbb{R}}

\def\M{\mathcal M}

\def\R{\mathbb{R}}

\begin{document}
\title [Markov processes of cubic stochastic matrices]
{Markov processes of cubic stochastic matrices:\\ {\it Quadratic stochastic processes}}

\author {J.M. Casas, M. Ladra, U.A. Rozikov}

\address{J.M.~Casas\\ Departamento Matem\'atica Aplicada I,
Universidade de Vigo,
E. E. Forestal, Campus Universitario A Xunqueira, 36005 Pontevedra, Spain.}
 \email {jmcasas@uvigo.es}

\address{M.~Ladra\\ Departamento de  Matem\'aticas, University of Santiago de Compostela, 15782, Spain.}
 \email {manuel.ladra@usc.es}

 \address{U.A.~Rozikov \\ Institute of Mathematics, 29, Do'rmon Yo'li str., 100125,
Tashkent, Uzbekistan.} \email {rozikovu@yandex.ru}

\begin{abstract} We consider Markov processes of cubic
stochastic (in a fixed sense) matrices which
are also called  quadratic stochastic process (QSPs).
A QSP is a particular case of a continuous-time dynamical system
whose states are stochastic cubic matrices
satisfying an analogue of the Kolmogorov-Chapman equation (KCE).
Since there are several kinds of multiplications between
cubic matrices we have to fix first a multiplication
 and then consider the KCE with respect
to the fixed multiplication. Moreover, the notion of stochastic
cubic matrix also varies depending on the real
models of application. The existence of a stochastic (at each time)
solution to the KCE provides
the existence of a QSP.
In this paper, our aim is to construct QSPs
for two specially chosen notions of stochastic cubic matrices
and two multiplications of such matrices (known as Maksimov's multiplications).
We construct a wide class of QSPs and
give some time-dependent behavior of such
processes. We give an example with  applications to the Biology, constructing a QSP which
describes the time behavior (dynamics) of a population with
the possibility of twin births.
\end{abstract}



\maketitle

\section{Introduction}

A Markov process is a random process indexed by time, in which the future is independent of the past, given the present. Thus, Markov processes are the natural stochastic analogs of the deterministic processes described by differential and difference equations. They form one of the most important classes of random processes.
If the time space is $T=[0,\infty)$ and the state space is discrete,
then Markov processes are known as continuous-time Markov chains.

The Kolmogorov-Chapman equation (KCE) gives the fundamental
relationship between the probability transitions (kernels).
Namely, it is known that (see e.g. \cite{SK}) if
each element of a family of matrices satisfying the KCE
is stochastic, then it generates a Markov process.

There are many random processes which can not be described by Markov processes
of square stochastic matrices (see for example \cite{D,ht,K,Mak}).

To have non-Markov process one can consider a solution of the KCE which is
not stochastic for some time as in \cite{CLR,ORT,RM1},
where a chain of evolution algebras (CEA) is introduced and investigated.
Later, this notion of CEA was generalized in \cite{LR},
where a concept of flow of arbitrary finite-dimensional algebras (i.e.
their matrices of structural constants are \textsl{cubic} matrices) is introduced.

By Maksimov \cite{Mak} some associative multiplication rules of cubic matrices
as well as cubic analogues of stochastic or doubly stochastic square
matrices are introduced, for which he suggests several possible
probability interpretations.
Moreover,  the concept of a Markov interaction process (MIP) is defined.
It is shown that there exists a one-to-one correspondence between the
transition matrices defining a MIP and the stochastic cubic matrices
of a certain kind.

In this paper we study Markov process of cubic matrices,
which is a two-parametric family of
cubic stochastic matrices (we fix a notion of stochastic
 matrix and fix a multiplication rule of cubic matrices) satisfying the KCE.

The paper is organized as follows. In Section~\ref{S:definitions} we give the main definitions related to
Markov processes, cubic matrices, several kinds of multiplications of cubic matrices
and Markov processes of cubic matrices which are also called  quadratic stochastic processes (QSPs).
In Section~\ref{Se3} we describe all QSPs of type $(3|0)$, these are solutions of the KCE
in the class of $3$-stochastic cubic matrices, with respect to the Maksimov's $0$-multiplication
(see Section~\ref{S:definitions} for definitions).
In Section~\ref{Se4} we construct some QSPs of type $(12|a_0)$, these are solutions of the KCE
in the class of $(1,2)$-stochastic cubic matrices, with respect to the Maksimov's $a_0$-multiplication.
In Section~\ref{bio} we give an application of a QSP of type $(12|a_0)$ to a population
with a possibility of twins birth.
For several QSPs we study time-dependent behavior of the processes.

\section{Preliminaries}\label{S:definitions}
\subsection{Markov process of square matrices}\label{ss1} Let us recall first the notion of Markov process for square
stochastic matrices. This will be useful to compare with Markov processes of cubic matrices.

A square matrix $\mathcal U=(U_{ij})_{i,j=1}^m$ is called \emph{right stochastic}
if
\[U_{ij}\geq 0, \quad  \forall i,j=1,\dots,m; \qquad \sum_{j=1}^mU_{ij}=1, \ \ \forall i=1,\dots,m.\]

Similarly one can define a \emph{left stochastic} matrix being a non-negative real square matrix, with each column summing to 1 and
a \emph{doubly stochastic} matrix being a square matrix of non-negative real numbers with each row and column summing to 1.

A family of stochastic matrices $\{\mathcal U^{[s,t]}: s,t\geq 0\}$
is called a \emph{Markov process} if it satisfies the Kolmogorov-Chapman equation:
\begin{equation}\label{KC0}
\mathcal U^{[s,t]}=\mathcal U^{[s,\tau]}\mathcal U^{[\tau,t]}, \qquad \text{for all} \ \ 0\leq s<\tau<t.
\end{equation}
 Let $I=\{1,2,\dots,m\}$. A \emph{distribution} (or \emph{state}) of the set $I$  is a probability measure
 $x=(x_1,\dots,x_m)$, where $x_i$ is a probability of $i\in I$. The set of all such vectors is called a simplex and denoted by
\[S^{m-1}=\left\{x\in\mathbb R^m: x_i\geq 0, \ \sum_{i=1}^mx_i=1\right\}.\]
Let $x^{(0)}=(x_1^{(0)}, \dots, x_m^{(0)})\in S^{m-1}$ be an initial distribution on $I$.
Denote by $x^{(t)}=(x_1^{(t)}, \dots, x_m^{(t)})\in S^{m-1}$ the distribution of the system at the moment $t$.
For arbitrary moments of time $s$ and $t$ with $s<t$ the matrix $\mathcal U^{[s,t]}=\left(U^{[s,t]}_{ij}\right)$ gives the transition probabilities
from the distribution $x^{(s)}$ to the distribution $x^{(t)}$. Moreover $x^{(t)}$ depends linearly from $x^{(s)}$:
\[x^{(t)}_k=\sum_{i=1}^mU_{ik}^{[s,t]}x^{(s)}_i, \qquad k=1,\dots,m.\]
A Markov chain is a type of Markov process that has either discrete state space or discrete time,
but the precise definition of a Markov chain varies (see e.g. \cite{A,Sh,SK} for the theory of Markov process).
\subsection{Cubic matrices}
We consider a cubic matrix $Q=(q_{ijk})_{i,j,k=1}^m$ as a $m^3$-dimensional vector, i.e. an element of $\mathbb R^{m^3}$,
which can be uniquely written as
\[Q=\sum_{i,j,k=1}^m   q_{ijk}E_{ijk},\]
where $E_{ijk}$ denotes the cubic unit (basis) matrix, i.e. $E_{ijk}$ is a $m^3$- cubic matrix whose
$(i,j,k)$th entry is equal to 1 and all the other entries are equal to 0.

Denoting $Q_i=(q_{ijk})_{j,k=1}^m$ we can write the cubic matrix $Q$ in the following form
\[Q=(Q_1| Q_2 |\dots |Q_m).\]

Denote by $\mathfrak C$ the set of all cubic matrices over a field $F$. Then $\mathfrak C$ is an
$m^3$-dimensional vector space over $F$, i.e. for any matrices $A=(a_{ijk})$, $B=(b_{ijk})\in \mathfrak C$, $\lambda\in F$, we have
\[ A+B \coloneqq (a_{ijk}+b_{ijk})\in \mathfrak C, \qquad  \lambda A \coloneqq (\lambda a_{ijk})\in \mathfrak C. \]
In general, one can fix an $m^3\times m^3\times m^3$- cubic matrix $\mu=\left(C_{ijk,lnr}^{uvw}\right)$ as a matrix of structural constants and   give a \emph{multiplication} of basis cubic matrices as
\begin{equation}\label{uk}
E_{ijk}*_\mu E_{lnr}=\sum_{uvw}C_{ijk,lnr}^{uvw}E_{uvw}.
\end{equation}
Then the extension of this multiplication by bilinearity to arbitrary cubic matrices gives a general
multiplication on the set $\mathfrak C$ and it becomes an algebra of cubic matrices (ACM), denoted by $\mathfrak C_\mu$ (see \cite{LRl} for some basic properties of ACM).
Under known conditions (see \cite{Ja}) on structural constants one can make this general ACM as a commutative or/and associative algebra, etc.

\subsection{Maksimov's multiplications} Introduce some simple versions of multiplications \eqref{uk}.
Denote $I=\{1,2,\dots,m\}$.

Following  \cite{Mak} define the following multiplications for basis matrices $E_{ijk}$:
\begin{equation}\label{ma0}
 E_{ijk}*_0 E_{lnr}=\delta_{kl}\delta_{jn}E_{ijr}.
 \end{equation}
Then for any two cubic matrices $A=(a_{ijk}), B=(b_{ijk})\in \mathfrak C$
 the matrix $A*_0 B=(c_{ijk})$ is defined by
 \begin{equation}\label{AB0}
 c_{ijr}=\sum_{k=1}^m a_{ijk}b_{kjr}.
 \end{equation}
Consider also
\begin{equation}\label{ma}
 E_{ijk}*_a E_{lnr}=\delta_{kl}E_{ia(j,n)r},
 \end{equation}
  where $a \colon   I\times I\to I$, $ (j,n) \mapsto a(j,n) \in I$,  is an arbitrary associative binary operation and
  $\delta_{kl}$ is the Kronecker symbol.
  Note that \eqref{ma0} is  \textsl{not} a particular case of \eqref{ma}.

  Denote by $\mathcal O_m$ the set of all associative binary operations on $I$.

 The general formula for the multiplication is the extension of \eqref{ma} by bilinearity, i.e.
 for any two cubic matrices $A=(a_{ijk}), B=(b_{ijk})\in \mathfrak C$
 the matrix $A*_a B=(c_{ijk})$ is defined by
\[
 c_{ijr}=\sum_{l,n: \, a(l,n)=j}\sum_k a_{ilk}b_{knr}.
\]

 Denote by $\mathfrak C_a\equiv \mathfrak C_a^m=(\mathfrak C, *_a) $, $a\in\mathcal O_m$,
 the ACM given by the multiplication $*_a$.

\subsection{Markov process as a quadratic stochastic process}
Following \cite{LR} we define a quadratic stochastic process.

Define several kinds of cubic stochastic matrices (see \cite{Mak,MG}):
a cubic matrix $P=(p_{ijk})_{i,j,k=1}^m$ is called
\begin{itemize}
\item[] $(1,2)$-\emph{stochastic} if
\[p_{ijk}\geq 0, \qquad  \sum_{i,j=1}^mp_{ijk}=1, \ \ \text{for all} \ k.\]

\item[] $(1,3)$-\emph{stochastic} if
\[p_{ijk}\geq 0, \qquad  \sum_{i,k=1}^mp_{ijk}=1, \ \ \text{for all} \ j.\]

\item[] $(2,3)$-\emph{stochastic} if
\[p_{ijk}\geq 0, \qquad  \sum_{j,k=1}^mp_{ijk}=1, \ \ \text{for all} \ i.\]

\item[] $3$-\emph{stochastic} if
\[p_{ijk}\geq 0, \qquad  \sum_{k=1}^mp_{ijk}=1, \ \ \text{for all} \  i,j.\]
The last one can be also given with respect to first and second index.
\end{itemize}

Maksimov \cite{Mak} also defined a twice stochastic matrix:
a (2,3)-stochastic cubic matrix is called \emph{twice stochastic} if
\[\sum_{i=1}^m p_{ijk}=\frac{1}{m}, \qquad \text{for all} \ \ j,k.\]
Denote by $\mathcal S$ the set of all possible kinds of stochasticity and
denote by $\mathbb M$ the set
of all possible multiplication rules of cubic matrices.

Let parameters $s\geq 0$, $t\geq 0$, are considered as time.

Denote by
$\M^{[s,t]}=\left(P_{ijk}^{[s,t]}\right)_{i,j,k=1}^{\ \underset{m}{}}$ a cubic matrix with two parameters.
\begin{defn}[\cite{LR}]
A family $\{\M^{[s,t]}: \ s,t\in\R_+\}$ is called a Markov process of cubic matrices (or a quadratic stochastic process (QSP))
of type $(\sigma|\mu)$ if for each time $s$ and $t$ the cubic matrix
$\M^{[s,t]}$ is stochastic in sense $\sigma\in\mathcal S$ and satisfies the Kolmogorov-Chapman equation
(for cubic matrices):
\begin{equation}\label{KC}
\M^{[s,t]}=\M^{[s,\tau]}*_\mu\M^{[\tau,t]}, \qquad \text{for all} \ \ 0\leq s<\tau<t.
\end{equation} with respect to the multiplication $\mu\in \mathbb M$.
\end{defn}

We note that this definition of QSP gives an alternative of   \cite[Definition 3.1.1]{MG} and a natural
generalization of the Markov process of Subsection~\ref{ss1}.

In \cite{LR} using the QSPs some flows of finite-dimensional algebras are determined and investigated.

\subsection{Motivations and interpretations}
QSPs arise naturally in the study of biological and physical systems with interactions.
Indeed, assume a particle of type $i$ and a particle of type $j$ have interaction at time $s$, as an interaction process, then
with probability $P_{ijk}^{[s,t]}$ a particle of type $k$ appears at time $t$. The Kolmogorov-Chapman equation \eqref{KC} gives
the time-dependent evolution law of the interacting process (dynamical system).

Let $x^{(0)}=(x_1^{(0)}, \dots, x_m^{(0)})\in S^{m-1}$ be an initial distribution on $I$.

Denote by $x^{(t)}=(x_1^{(t)}, \dots, x_m^{(t)})\in S^{m-1}$ the distribution of the system at the moment $t$.
For arbitrary moments of time $s$ and $t$, with $s<t$, the matrix $\M^{[s,t]}$ gives the transition probabilities
from the distribution $x^{(s)}$ to the distribution $x^{(t)}$.

Since we should have $x^{(t)}\in S^{m-1}$, one can consider the following models:
\begin{itemize}
\item[-] Consider $P_{ijk}^{[s,t]}$ as the conditional
probability $P^{[s,t]}(k|i, j)$ that $i$th and $j$th particles (physics) or species (biology)
interbred successfully at time $s$, then they produce
an individual $k$ at time $t$.

Assume the ``parents'' $ij$ are independent
for any moment of time $s$ and the matrix $\left(P_{ijk}^{[s,t]}\right)$ is
3-stochastic, then the probability distribution $x^{(t)}$
can be found by the total probability as
\[
x^{(t)}_k=\sum_{i,j=1}^mP_{ijk}^{[s,t]}x^{(s)}_ix^{(s)}_j, \qquad  k=1, \dots, m, \ \ 0\leq s<t.
\]
For $1-$stochastic and $2-$stochastic it can be defined similarly, by replacing the corresponding indices.

\item[-] Consider now a physical (biological, chemical) system where
there are $m$ types of ``particles'' or molecules, the set of types is denoted by $I=\{1,\dots,m\}$,
 and each particle may split to two new ones
 having types from $I$. Consider $P_{ijk}^{[s,t]}$ as the conditional
probability $P^{[s,t]}(i,j|k)$ that a particle of type $k$ starts splitting
at time $s$ and finishes splitting at time $t$ and the result is
two particles with $i$th and $j$th types.
For a biological model see Section~\ref{bio}.

Assume  $\left(P_{ijk}^{[s,t]}\right)$ is (1,2)-stochastic then $x^{(t)}$ can be defined by
\begin{equation}\label{xt12}
x^{(t)}_k= \frac{1}{2} \sum_{i,j=1}^m\left(P_{kij}^{[s,t]}+P_{ikj}^{[s,t]}\right)x^{(s)}_j, \qquad k=1, \dots, m, \ \ 0\leq s<t.
\end{equation}
For (1,3)-stochastic and (2,3)-stochastic cases one can define similarly by replacing the indices.
\end{itemize}

Thus finding $P_{ijk}^{[s,t]}$ from the equation \eqref{KC} (at a fixed $(\sigma,\mu)$) and studying the time-dependent behavior
of $P_{ijk}^{[s,t]}$ we can describe the time-dependent evolution of $x^{(t)}$.

\begin{defn}
 A QSP is called a (time) homogenous if
the matrix $\M^{[s,t]}$ depends only on $t-s$. In this case we write
$\M^{[t-s]}$.
\end{defn}
\begin{defn}
 A QSP is called periodic if
the matrix $\M^{[s,t]}$ depends on time $s$ or/and $t$  periodically, i.e.
periodicity with respect to $s$ (resp. $t$): there is $S>0$ (resp. $T>0$), such that
$\M^{[s+S,t]}=\M^{[s,t]}$ for all $0\leq s<s+S<t$ (resp. $\M^{[s,t+T]}=\M^{[s,t]}$, for all $0\leq s<t$).
\end{defn}

\subsection{Our aim}
To construct a QSP of type $(\sigma, \mu)$ one has to solve \eqref{KC}. In this paper our aim is to study QSPs, for the
following two cases
\begin{itemize}
\item[-]$\sigma=3$-stochastic and $\mu$ is the Maksimov's multiplication $\mu=0$ given by \eqref{ma0}.
 Call this QSP of type $(3|0)$.
Under these conditions the equation \eqref{KC} has the following form
\begin{equation}\label{e0}
P_{ijr}^{[s,t]}=\sum_{k=1}^m P^{[s,\tau]}_{ijk}P_{kjr}^{[\tau,t]}, \quad \forall i,j,r\in I,
\end{equation}
where
\begin{equation}\label{e3}
P_{ijk}^{[s,t]}\geq 0, \ \ \forall i,j,k\in I; \qquad  \sum_{k=1}^nP_{ijk}^{[s,t]}=1, \ \ \text{for all} \ \ i,j\in I, \ \ 0\leq s<t.
\end{equation}
Thus a QSP of type $(3|0)$ is a solution to the system \eqref{e0} and \eqref{e3}.

\item[-]  $\sigma=(1,2)$-stochastic and $\mu$ is the Maksimov's multiplication
with the operation $a=a_0$ such that $a_0(i,j)=i$ for any $i,j\in I$.
Call this QSP of type $(12|a_0)$.
Under these conditions the equation \eqref{KC} has the following form
\begin{equation}\label{eD}
P_{ijr}^{[s,t]}=\sum_{k,n=1}^m P^{[s,\tau]}_{ijk}P_{knr}^{[\tau,t]}, \quad \forall i,j,r\in I,
\end{equation}
where
\begin{equation}\label{e12}
P_{ijr}^{[s,t]}\geq 0, \ \ \forall i,j,r\in I; \qquad \sum_{i,j=1}^mP_{ijr}^{[s,t]}=1, \ \ \forall r\in I, \ \ 0\leq s<t.
\end{equation}
Hence a QSP of type $(12|a_0)$ is a solution to the system \eqref{eD} and \eqref{e12}.
\end{itemize}

\section{QSPs of type $(3|0)$}\label{Se3}

  For the multiplication \eqref{AB0} it is easy to see that if two cubic matrices, say $A$ and $B$, are 3-stochastic then
  their multiplication $A*_0B$ is  3-stochastic too.

  Let $\mathcal M_j^{[s,t]}=(P^{[s,t]}_{ijk})_{i,k=1}^m$ be the $j$th layer of the matrix $\mathcal M^{[s,t]}$.
  The following proposition characterizes all QSPs of type (3|0).

\begin{pro}[\cite{LR}]\label{pm} Any solution of the equation \eqref{KC} for the multiplication \eqref{ma0} is
a direct sum of solutions of the following $m$ independent equations:
\[
\M_j^{[s,t]}=\M_j^{[s,\tau]}\M_j^{[\tau,t]}, \qquad \text{for all} \quad 0\leq s<\tau<t, \quad j=1,\dots,m.
\]
\end{pro}
The following lemma is obvious
\begin{lemma}\label{l1} The matrix $\M^{[s,t]}$ is 3-stochastic if and only if the square matrix $\M^{[s,t]}_j$ is right stochastic for any $j=1,\dots,m$.
\end{lemma}
As corollary of Proposition~\ref{pm} and Lemma~\ref{l1} we have the following.
\begin{thm}\label{t1} Any QSP of type (3|0) is a direct sum of $m$ right stochastic square matrices
 $\mathcal M_j^{[s,t]}=(P^{[s,t]}_{ijk})_{i,k=1}^m$ satisfying  equation \eqref{KC0}.
 Consequently, any QSP of type (3|0) consists $m$ independent collection of usual Markov processes (see Subsection~\ref{ss1}).
\end{thm}
The independence mentioned in Theorem~\ref{t1} allows us to say that the QSPs of type (3|0) are not interesting,
because the basic theory of Markov process of square matrices is well developed.

Here we give examples of QSPs of type (3|0). This example also will be used to construct QSPs of type $(12|a_0)$.
\begin{ex}\label{ex1} In \cite{RM} to construct chains of some algebras, for $m=2$,
 a wide class of solutions of \eqref{KC0} is presented,
many of them are non-stochastic matrices, in general.
Here we list the following families of (left, right, doubly) {\rm stochastic}
square matrices (see \cite{RM}), which
satisfy the equation \eqref{KC0}, i.e. they generate independently interesting Markov processes:
\begin{align*}
 \Q_1^{[s,t]}&=\begin{pmatrix}
g(s)& g(s)\\[2mm]
1-g(s) &1-g(s)\\[2mm]
\end{pmatrix}, \  \text{where} \ \ g(s)\in [0,1] \  \text{is an arbitrary function};\\
\Q_2^{[s,t]}&=\frac{1}{2} \begin{pmatrix}
1+\frac{\Psi(t)}{\Psi(s)} & 1-\frac{\Psi(t)}{\Psi(s)}\\[2mm]
1-\frac{\Psi(t)}{\Psi(s)} & 1+\frac{\Psi(t)}{\Psi(s)}\\[2mm]
\end{pmatrix},
\end{align*}
 where $\Psi(t)> 0$ is an arbitrary decreasing function of $t\geq 0$;
  \begin{align*}
  \Q_3^{[s,t]}&=
\begin{cases}
\ \ \begin{pmatrix}
1 & 0\\[2mm]
0 & 1 \\[2mm]
\end{pmatrix}, & \ \ \text{if} \ \ s\leq t<b,\\[4mm]
\frac{1}{2} \begin{pmatrix}
1 & 1\\[2mm]
1 & 1 \\[2mm]
\end{pmatrix},& \ \ \text{if} \ \ t\geq b,\\
\end{cases}, \quad \text{where} \ \ b>0 ;\\[3mm]
 \Q_4^{[s,t]}&=\begin{pmatrix}
1  & 0\\[2mm]
1-\frac{\psi(t)}{\psi(s)} & \frac{\psi(t)}{\psi(s)}\\[2mm]
\end{pmatrix},
\end{align*}
where $\psi(t)>0$ is a decreasing function
of $t\geq 0$;
\begin{align*}
 \Q_5^{[s,t]}&=\begin{pmatrix}
f(t)& 1-f(t)\\[2mm]
f(t) &1-f(t)\\[2mm]
\end{pmatrix}, \ \ \text{where} \ \ f(t)\in [0,1] \ \ \text{is an arbitrary function};\\
 \Q_6^{[s,t]}(\lambda,\mu)&=
\begin{pmatrix}
1-\frac{\lambda-2\mu}{2(\lambda-\mu)}\left(1- \frac{\theta(t)}{\theta(s)} \right)&
\frac{\lambda-2\mu}{ 2(\lambda-\mu)} \left(1-\frac{\theta(t)}{\theta(s)}  \right)\\[2mm]
\frac{\lambda}{2(\lambda-\mu)} \left(1-\frac{\theta(t)}{\theta(s)} \right)&
 1-\frac{\lambda}{2(\lambda-\mu) }\left(1- \frac{\theta(t)}{\theta(s)}\right) \\[2mm]
\end{pmatrix},
\end{align*}
where  $\lambda$, $\mu$ are real parameters  such that $0<2\mu<\lambda$ and $\theta(t)>0$ is an arbitrary decreasing function;
\[\Q_7^{[s,t]}=
\begin{cases}
\begin{pmatrix}
1 & 0\\[2mm]
0 & 1\\[2mm]
\end{pmatrix}, & \ \ \text{if} \ \ s\leq t<a,\\[4mm]
\begin{pmatrix}
g(t) & 1-g(t)\\[2mm]
g(t) & 1-g(t)\\[2mm]
\end{pmatrix}, & \ \ \text{if} \ \ t\geq a,\\
\end{cases} \  \text{where} \ \ g(t)\in [0,1] \  \text{is an arbitrary function}.\]

Using the right stochastic matrices we can construct the following QSPs of type (3|0):
\[\M^{[s,t]}=\left(\M_1^{[s,t]}\, \big| \, \M_2^{[s,t]}\right), \ \ \text{with any} \ \ \M_1^{[s,t]}, \M_2^{[s,t]}
\in \left\{\Q_2^{[s,t]}, \Q_3^{[s,t]}, \dots, \Q_7^{[s,t]}\right\}.\]

We note that the matrices $\Q_i^{[s,t]}$, $i=1,\dots,7$, generate interesting usual Markov processes:
some of them independent on time, some depend only on $t$,
but many of them non-homogenously depend on both $s,t$.  Depending on the statistical models
of real-world processes one can choose
parameter functions (i.e. $g$,  $\Psi$, $\psi$, $f$,  $\theta$) and be able then to control
the evolution (with respect to time) of such Markov processes.
Then the evolution of the QSP of type (3|0) will be given by the evolution of two independent Markov processes.
\end{ex}

\section{QSPs of type $(12|a_0)$}\label{Se4}

Let $\mathcal M^{[s,t]}=\left(P_{ijk}^{[s,t]}\right)$ be a cubic matrix,
 define the square matrix  $\overline{\mathcal M}^{[s,t]}=(\bar c^{[s,t]}_{ik})$ with
\begin{equation}\label{cd}
\bar c^{[s,t]}_{ik}=\sum_{j=1}^m  P_{ijk}^{[s,t]}, \qquad  i,k=1,\dots,m.
\end{equation}

\begin{pro}[\cite{LR}] Any solution of equation \eqref{KC} for the multiplication of type $a_0$ (equivalently equation \eqref{eD})
can be given by a solution of the system \eqref{cd} with a matrix $\overline{\mathcal M}^{[s,t]}=(\bar c^{[s,t]}_{ik})$ which satisfies
\eqref{KC0}.
\end{pro}
From this proposition it follows that the family of matrices $\overline{\mathcal M}^{[s,t]}=(\bar c^{[s,t]}_{ik})$
is a Markov process if and only if the matrices are left stochastic.

The following lemma gives a connection between left stochastic and (1,2)-stochastic matrices.
\begin{lemma}\label{l2} The matrix $\mathcal M^{[s,t]}=\left(P_{ijk}^{[s,t]}\right)$, with $P_{ijk}^{[s,t]}\geq 0$,  is (1,2)-stochastic if and only if the corresponding  matrix $\overline{\mathcal M}^{[s,t]}$
is left stochastic.
\end{lemma}
\begin{proof} It is consequence of the equality \eqref{cd}.
\end{proof}

\subsection{Two-dimensional cases}
Now we construct QSPs of type $(12|a_0)$ corresponding to the left stochastic matrices mentioned in Example~\ref{ex1}.
Write a cubic matrix $\M^{[s,t]}$ for $m=2$ in the following convenient form:
\begin{equation}\label{edd}
\M^{[s,t]}= \begin{pmatrix}
P_{111}^{[s,t]} &P_{112}^{[s,t]}& \vline &P_{211}^{[s,t]} &P_{212}^{[s,t]}\\[3mm]
P_{121}^{[s,t]} &P_{122}^{[s,t]}& \vline &P_{221}^{[s,t]} &P_{222}^{[s,t]}
\end{pmatrix}.
\end{equation}

{\bf Case $\Q^{[s,t]}_1$:} Let $\overline{\mathcal M}^{[s,t]}=\Q_1^{[s,t]}$. Then from \eqref{eD} by \eqref{cd}
we get
\begin{align*}
P_{ij1}^{[s,t]} & =g(s)P_{ij1}^{[s,\tau]}+(1-g(s))P_{ij2}^{[s,\tau]},   \qquad  i,j=1,2, \\
P_{ij2}^{[s,t]} & =g(s)P_{ij1}^{[s,\tau]}+(1-g(s))P_{ij2}^{[s,\tau]}, \qquad i,j=1,2.
\end{align*}
Consequently $P_{ij1}^{[s,t]}=P_{ij2}^{[s,t]}$. Therefore, by the last system we have $P_{ij1}^{[s,t]}=P_{ij1}^{[s,\tau]}$.
Hence   $P_{ij1}^{[s,t]}$ should not depend on $t$, i.e. there exists a function $u_{ij}(s)$
such that
\begin{equation}\label{cs}
 P_{ij1}^{[s,t]}=P_{ij2}^{[s,t]}=u_{ij}(s).
\end{equation}
By \eqref{cd} and \eqref{cs} we shall have
\begin{align*}
P_{111}^{[s,t]}+P_{121}^{[s,t]} &=P_{112}^{[s,t]}+P_{122}^{[s,t]}=u_{11}(s)+u_{12}(s)=g(s),\\
P_{211}^{[s,t]}+P_{221}^{[s,t]}& =P_{212}^{[s,t]}+P_{222}^{[s,t]}=u_{21}(s)+u_{22}(s)=1-g(s).
\end{align*}
Consequently the matrix \eqref{edd} has the following form:
\begin{equation}\label{esd}
\M^{[s,t]}_{(1)}=\begin{pmatrix}
u_{11}(s) &u_{11}(s)& \vline &u_{21}(s) & u_{21}(s) \\[3mm]
g(s)-u_{11}(s) &g(s)-u_{11}(s)& \vline &1-g(s)-u_{21}(s) &1-g(s)-u_{21}(s)
\end{pmatrix},
\end{equation}
where  $u_{11}$ and $u_{21}$ are arbitrary functions of $s\geq 0$.

Thus we proved the following.
\begin{pro} Let $g(s)\in [0,1]$ be an arbitrary function.
The family of matrices \eqref{esd}, $\M^{[s,t]}_{(1)}$,
is a QSP of type $(12|a_0)$ if and only if the functions $u_{11}(s)$ and $u_{21}(s)$ are such that
\[0\leq u_{11}(s)\leq g(s), \qquad  0\leq u_{21}(s)\leq 1-g(s).\]
\end{pro}

For this QSP $\M^{[s,t]}_{(1)}$, using \eqref{xt12}, let us give the time
behavior of the distribution $x^{(t)}=(x_1^{(t)}, x_2^{(t)})\in S^1$. Fix $s\geq 0$
and by taking a vector $x^{(s)}=(x_1^{(s)}, x_2^{(s)})\in S^1$, then by formula \eqref{xt12}
independently on the vector $x^{(s)}$, for \textsl{any} $t>s$,  we get
\begin{align*}
x_1^{(t)}& =A(s)\equiv  \frac{1}{2} \big(g(s)+u_{11}(s)+u_{21}(s)\big),\\[2mm]
x_2^{(t)}& =1-A(s)=1-\frac{1}{2}\big(g(s)+u_{11}(s)+u_{21}(s)\big).
\end{align*}
Thus the time behavior of $x^{(t)}$ is clear: start process at time $s$ with an arbitrary initial
distribution vector $x^{(s)}$ then as soon as
the time $t$ turns on the distribution of the system goes to
the distribution $(A(s), 1-A(s))$ and this distribution remains stable
during all time $t>s$.

{\bf Case $\Q^{[s,t]}_2$:} Let $\overline{\mathcal M}^{[s,t]}=\Q_2^{[s,t]}$. Then from \eqref{eD} by \eqref{cd}
we get
\begin{align*}
P_{ij1}^{[s,t]} & =\frac{1}{2} P_{ij1}^{[s,\tau]}\left(1+\frac{\Psi(t)}{\Psi(\tau)} \right)+\frac{1}{2} P_{ij2}^{[s,\tau]}
\left(1- \frac{\Psi(t)}{\Psi(\tau)}\right),   \qquad  i,j=1,2, \\
P_{ij2}^{[s,t]} & ={\frac{1}{2}}P_{ij1}^{[s,\tau]}\left(1-\frac{\Psi(t)}{\Psi(\tau)}\right)+
\frac{1}{2}P_{ij2}^{[s,\tau]}\left(1+\frac{\Psi(t)}{\Psi(\tau)}\right), \qquad i,j=1,2.
\end{align*}
Denoting  $\alpha_{ij}(s,t)=P_{ij1}^{[s,t]}-P_{ij2}^{[s,t]}$ and  $\beta_{ij}(s,t)=P_{ij1}^{[s,t]}+P_{ij2}^{[s,t]}$, from the last system of equations we get

\[ \frac{ \alpha_{ij}(s,t)}{\Psi(t)} = \frac{\alpha_{ij}(s,\tau)}{\Psi(\tau)}, \qquad \qquad
\beta_{ij}(s,t)=\beta_{ij}(s,\tau).\]

It follows from the last equalities that $\frac{\alpha_{ij}(s,t)}{\Psi(t)}$ and $\beta_{ij}(s,t)$  do not depend on $t$, i.e.
there are functions $\gamma_{ij}(s)$ and $\zeta_{ij}(s)$
such that
\[
 \alpha_{ij}(s,t)=\gamma_{ij}(s)\Psi(t), \qquad  \qquad \beta_{ij}(s,t)=\zeta_{ij}(s).
\]
Consequently,
\begin{equation}\label{alb}
P_{ij1}^{[s,t]}=\frac{1}{2} \left(\gamma_{ij}(s)\Psi(t)+\zeta_{ij}(s)\right),
\qquad \qquad P_{ij2}^{[s,t]}=\frac{1}{2} \left(\zeta_{ij}(s)-\gamma_{ij}(s)\Psi(t)\right).
\end{equation}
By these equalities from \eqref{cd} (for $\overline{\mathcal M}^{[s,t]}=\Q_2^{[s,t]}$) we get
\begin{align*}
\left(\gamma_{11}(s)+\gamma_{12}(s)- \frac{1}{\Psi(s)}\right)\Psi(t)+\zeta_{11}(s)+\zeta_{12}(s)&=1,\\
\zeta_{11}(s)+\zeta_{12}(s)-\left(\gamma_{11}(s)+\gamma_{12}(s)-\frac{1}{\Psi(s)}\right)\Psi(t)&=1,\\
\left(\gamma_{21}(s)+\gamma_{22}(s)+\frac{1}{\Psi(s)} \right)\Psi(t)+\zeta_{21}(s)+\zeta_{22}(s)&=1,\\
\zeta_{21}(s)+\zeta_{22}(s)-\left(\gamma_{21}(s)+\gamma_{22}(s)+\frac{1}{\Psi(s)} \right)\Psi(t)&=1.
\end{align*}
From this system we obtain
\begin{align*}
\zeta_{11}(s)+\zeta_{12}(s)=\zeta_{21}(s)+\zeta_{22}(s)&=1,\\
\gamma_{11}(s)+\gamma_{12}(s)=-\big(\gamma_{21}(s)+\gamma_{22}(s)\big)&=\frac{1}{\Psi(s)}.
 \end{align*}
 Using these equalities and \eqref{alb} the matrix \eqref{edd} can be written in the following form:
\begin{multline}\label{m2}
\M^{[s,t]}_{(2)}=\frac{1}{2} \left(\begin{array}{cc}
\zeta_{11}(s)+\gamma_{11}(s)\Psi(t) &\zeta_{11}(s)-\gamma_{11}(s)\Psi(t)\\[3mm]
1-\zeta_{11}(s)+\left(\frac{1}{\Psi(s)} -\gamma_{11}(s)\right)\Psi(t)
&1-\zeta_{11}(s)-\left(\frac{1}{\Psi(s)} -\gamma_{11}(s)\right)\Psi(t)
\end{array}\right.\\[3mm]
\vline \left.\begin{array}{cc}
\zeta_{21}(s)+ \gamma_{21}(s)\Psi(t)&\zeta_{21}(s)-\gamma_{21}(s)\Psi(t)\\[3mm]
1-\zeta_{21}(s)-\left(\frac{1}{\Psi(s)} +\gamma_{21}(s)\right)\Psi(t)
&1-\zeta_{21}(s)+\left(\gamma_{21}(s)+\frac{1}{\Psi(s)} \right)\Psi(t)
\end{array}\right),
\end{multline}
where $\Psi(t)>0$ is a decreasing function, $\gamma_{11}$, $\gamma_{21}$, $\zeta_{11}$ and $\zeta_{21}$ are arbitrary functions of $s\geq 0$.
\begin{pro}\label{p4}
The family of matrices \eqref{m2}, $\M^{[s,t]}_{(2)}$ (with $\Psi(t)>0$  a decreasing function),
is a QSP of type $(12|a_0)$ if and only if for the functions $\gamma_{11}$, $\gamma_{21}$, $\zeta_{11}$ and $\zeta_{21}$ the following conditions hold:
\begin{align*}
\frac{1}{2}\left(\frac{1}{\Psi(s)} -\frac{1}{\Psi(t)}\right) & \leq\gamma_{11}(s)\leq \frac{1}{2} \left(\frac{1}{\Psi(s)}+\frac{1}{\Psi(t)}\right), \ \ \forall s,t, \ 0\leq s<t;\\
{} - \frac{1}{2} \left(\frac{1}{\Psi(s)} + \frac{1}{\Psi(t)}\right) & \leq\gamma_{21}(s)\leq \frac{1}{2} \left(\frac{1}{\Psi(s)}
-\frac{1}{\Psi(t)}\right), \ \ \forall s,t, \ 0\leq s<t;
\end{align*}
\[0\leq \zeta_{11}(s)\leq 1, \qquad  0\leq \zeta_{21}(s)\leq 1, \quad  \forall s\geq 0.\]
\end{pro}
\begin{proof} It is easy to see that the matrix $\M_{(2)}^{[s,t]}$ satisfies
 $\sum_{i,j}P_{ijk}^{[s,t]}=1$, for $k=1,2$.
Therefore we shall show that $P_{ijk}^{[s,t]}\geq 0$. The system of inequalities $P_{1ij}^{[s,t]}\geq 0$, $i,j=1,2$,
is equivalent to
\begin{align*}
0 & \leq  \zeta_{11}(s)+\gamma_{11}(s)\Psi(t)\leq 1+ \frac{\Psi(t)}{\Psi(s)},\\
 0 &\leq  \zeta_{11}(s)-\gamma_{11}(s)\Psi(t)\leq 1-\frac{\Psi(t)}{\Psi(s)}.
\end{align*}
Solving this system of inequalities with respect to $\zeta_{11}(s)$ and $\gamma_{11}(s)$ we get the conditions mentioned in the proposition.
The conditions for   $\zeta_{21}(s)$ and $\gamma_{21}(s)$ can be obtained similarly from the system of inequalities $P_{2ij}^{[s,t]}\geq 0$, $i,j=1,2$.
\end{proof}
\begin{rk} If $\Psi(t)>0$ is a bounded function, say $\Psi(0)\leq \Psi(t)\leq \Psi(\infty)$,
then the condition of Proposition~\ref{p4} can be given uniformly with respect to $t$, i.e.
one gets
\begin{align*}
\frac{1}{2} \left(\frac{1}{\Psi(s)}-\frac{1}{\Psi(\infty)}\right) &
 \leq\gamma_{11}(s)\leq \frac{1}{2} \left(\frac{1}{\Psi(s)} +\frac{1}{\Psi(\infty)} \right), \quad  \ \text{for all} \  s\geq 0; \\
{} - \frac{1}{2} \left(\frac{1}{\Psi(s)}+\frac{1}{\Psi(\infty)}\right) &
 \leq\gamma_{21}(s)\leq  \frac{1}{2} \left(\frac{1}{\Psi(s)}-\frac{1}{\Psi(0)}\right), \qquad \text{for all} \  s\geq 0.
\end{align*}
\end{rk}

Now let us give, for the QSP $\M^{[s,t]}_{(2)}$, the time
behavior of the distribution $x^{(t)}=(x_1^{(t)}, x_2^{(t)})\in S^1$. Fix $s\geq 0$
and by taking an initial distribution $x^{(s)}=(x_1^{(s)}, x_2^{(s)})\in S^1$, then by formula \eqref{xt12}
for \textsl{any} $t>s$,  we get
\begin{align*}
x_1^{(t)}& =\frac{1}{4} \left(1+\zeta_{11}(s)+\zeta_{21}(s)+\Big\{\gamma_{11}(s)+\gamma_{21}(s)+\frac{1}{\Psi(s)} \Big\}\Psi(t)\right)x_1^{(s)}\\[2mm]
 & {} +\frac{1}{4} \left(1+\zeta_{11}(s)+\zeta_{21}(s)-\Big\{\gamma_{11}(s)+\gamma_{21}(s)+\frac{1}{\Psi(s)} \Big\}\Psi(t)\right)x_2^{(s)},\\[3mm]
x_2^{(t)}& =\frac{1}{4} \left(3-\zeta_{11}(s)-\zeta_{21}(s)-\Big\{\gamma_{11}(s)+\gamma_{21}(s)+\frac{1}{\Psi(s)} \Big\}\Psi(t)\right)x_1^{(s)}\\[2mm]
 &  {} +\frac{1}{4}\left(3-\zeta_{11}(s)-\zeta_{21}(s)+\Big\{\gamma_{11}(s)+\gamma_{21}(s)+\frac{1}{\Psi(s)} \Big\}\Psi(t)\right)x_2^{(s)}.
\end{align*}
Thus the time behavior of $x^{(t)}$ depends on the function $\Psi(t)$:
\begin{itemize}
\item[-] If $\Psi(t)$ has a limit, say $\Psi(\infty)$, then depending on the initial vector $x^{(s)}$ we have
the following limit distribution:
\[\lim_{t\to\infty}x^{(t)}=\big(x_1^{(\infty)}(s), 1- x_1^{(\infty)}(s)\big),\]
where
\begin{align*}
x_1^{(\infty)}(s) &=\frac{1}{4} \left(1+\zeta_{11}(s)+\zeta_{21}(s)+\Big\{\gamma_{11}(s)+\gamma_{21}(s)+\frac{1}{\Psi(s)} \Big\}\Psi(\infty)\right)x_1^{(s)}\\
 &  {} +\frac{1}{4} \left(1+\zeta_{11}(s)+\zeta_{21}(s)-\Big\{\gamma_{11}(s)+\gamma_{21}(s)+\frac{1}{\Psi(s)} \Big\}\Psi(\infty)\right)x_2^{(s)}.
\end{align*}
\item[-] If $\Psi(t)$ is a periodic function, then for any $t>s$ the behavior
of $x^{(t)}$ will be periodic.
\end{itemize}
Thus if one starts the process at time $s$ with an arbitrary initial
distribution vector $x^{(s)}$ then the distribution of the system goes to
a limit distribution if $\Psi$ has a limit, otherwise, the set of limit points of
$x^{(s)}$ is equivalent to the set of limit points of $\Psi(t)$. Concluding, we say that
choosing the parameter functions $\Psi$, $\zeta_{11}$, $\zeta_{21}$, $\gamma_{11}$, and
$\gamma_{21}$, one can control the behavior of the distribution $x^{(t)}$ during all time $t>s$.

{\bf Case $\Q_3^{[s,t]}$:} Let $\overline{\mathcal M}^{[s,t]}=\Q_3^{[s,t]}$.
Then from \eqref{eD} by \eqref{cd} for $t<b$ we get
\[P_{ij1}^{[s,t]} =P_{ij1}^{[s,\tau]}, \qquad \quad P_{ij2}^{[s,t]} =P_{ij2}^{[s,\tau]},\quad i,j=1,2.\]
Consequently, there are $\eta_{ij}(s)$ and $\xi_{ij}(s)$ such that
\[P_{ij1}^{[s,t]} =\eta_{ij}(s), \qquad \quad P_{ij2}^{[s,t]} =\xi_{ij}(s),\quad i,j=1,2.\]
Moreover, by \eqref{cd}, for any $s\geq 0$, we shall have
\begin{align*}
\eta_{11}(s)+\eta_{12}(s)=1, \qquad \quad \eta_{21}(s)+\eta_{22}(s)=0.\\
\xi_{11}(s)+\xi_{12}(s)=0, \qquad  \quad \xi_{21}(s)+\xi_{22}(s)=1.
\end{align*}
In case $t\geq b$ the solution of the equation can be reduced to the case $\Q_1^{[s,t]}$ with $g(s)\equiv \frac{1}{2}$.
Therefore, we get
\[\M^{[s,t]}_{(3)}=
\begin{cases}
\begin{pmatrix}
\eta_{11}(s) & \xi_{11}(s) & \vline & \eta_{21}(s) & \xi_{21}(s)\\
1-\eta_{11}(s) & -\xi_{11}(s) & \vline & -\eta_{21}(s) & 1-\xi_{21}(s)
\end{pmatrix}, &  \text{if} \  s\leq t<b,\\[4mm]
\begin{pmatrix}
\kappa_{11}(s) & \kappa_{11}(s) & \vline & \kappa_{21}(s) & \kappa_{21}(s) \\
\frac{1}{2} -\kappa_{11}(s) &\frac{1}{2}-\kappa_{11}(s) & \vline & \frac{1}{2}-\kappa_{21}(s) & \frac{1}{2}-\kappa_{21}(s)
\end{pmatrix}, & \text{if}  \ t\geq b,\\
\end{cases}, \ \text{where}  \ b>0.
\]
The following proposition is obvious.

\begin{pro} The family of matrices $\M_{(3)}^{[s,t]}$ is a QSP of type $(12|a_0)$ if and only if
\[\eta_{11}(x),\, \xi_{21}(s)\in [0,1], \qquad \eta_{21}(s)=\xi_{11}(s)\equiv 0, \qquad  \kappa_{11}(x),\, \kappa_{21}(s)\in [0,1/2].\]
\end{pro}

\subsection{$m$-dimensional case}

For arbitrary $m$ the following theorem gives an example of time
non-homogenous QSP of type $(12|a_0)$:

\begin{thm}\label{tA}
 Let $\{A^{[t]}=(a_{ij}^{[t]}), \,t\geq 0\}$ be a family of
invertible  $m\times m$ square matrices (for all $t$), and
let $(A^{[t]})^{-1}=(b_{ij}^{[t]})$ denote the inverse of $A^{[t]}$.
Assume that
\begin{itemize}
\item[(i)] The square matrix $\overline\M^{[s,t]}=A^{[s]}(A^{[t]})^{-1}$ is left stochastic for any $s<t$.
\item[(ii)] Take arbitrary functions $\beta_{ijk}^{(s)}$, $i,j,k=1, \dots,m$, such that
\begin{align*}
\sum_{j=1}^m\beta_{ijk}^{(s)} & =a_{ik}^{[s]}, \qquad \text{for any} \ \ i,k \ \ \text{and} \ \ s,\\
\sum_{k=1}^m\beta_{ijk}^{(s)}b_{kr}^{[t]} & \geq 0, \qquad \text{for any} \ \ i,j,r \ \ \text{and} \ \ s<t.
\end{align*}
\end{itemize}
Then the cubic matrix
\begin{equation}\label{ml}
\M^{[s,t]}=\left(\sum_{k=1}^m\beta_{ijk}^{(s)}b_{kr}^{[t]}\right)_{i,j,r=1}^m
\end{equation}
generates a QSP of type $(12|a_0)$.
\end{thm}
\begin{proof} By  \cite[Theorem 1]{LR} it is known that the matrix \eqref{ml} satisfies the equation \eqref{KC} for
the multiplication $a_0$. By the conditions (i), (ii) and Lemma~\ref{l2} we conclude that the matrix \eqref{ml} is (1,2)-stochastic.
Thus this matrix satisfies conditions \eqref{eD} and \eqref{e12}, i.e. is a QSP of type $(12|a_0)$.
\end{proof}
Since the inverse of a stochastic matrix may not be stochastic,
one wants to have an example of a family of matrices $A^{[t]}$ satisfying
conditions of Theorem~\ref{tA}.
The following proposition gives such an example.
\begin{pro}\label{p3} Let $m=2$ and suppose that the matrix $A^{[t]}$, $t\geq 0$, has the form
\[A^{[t]}=\begin{pmatrix}
a(t)& 1-b(t)\\
 1-a(t)& b(t)
\end{pmatrix},\]
where $a(t), b(t)\in (0,1)$ are arbitrary increasing (resp. decreasing) functions such that
$a(t)+b(t)>1$ (resp. $a(t)+b(t)<1$), $\forall t\in (0,1)$. Then the matrix $A^{[t]}$ satisfies condition (i) of Theorem~\ref{tA}.
\end{pro}
\begin{proof} From condition $\det(A^{[t]})=a(t)+b(t)-1\ne 0$ it follows that  $A^{[t]}$ is invertible
for any $t\in (0,1)$. We shall prove that it satisfies the condition (i) of Theorem~\ref{tA}.
We have
\[(A^{[t]})^{-1}=\frac{1}{a(t)+b(t)-1} \begin{pmatrix}
b(t)& -1+b(t)\\
-1+a(t)& a(t)
\end{pmatrix}.\]
Using this equality we get
\begin{multline*}
A^{[s]}(A^{[t]})^{-1}=(\mathcal A^{[s,t]}_{ij})_{i,j=1,2}\\
=\frac{1}{a(t)+b(t)-1}\begin{pmatrix}
a(s)b(t)-(1-b(s))(1-a(t))& a(t)(1-b(s))-a(s)(1-b(t))\\[2mm]
b(t)(1-a(s))-b(s)(1-a(t))& a(t)b(s)-(1-a(s))(1-b(t))
\end{pmatrix}.
\end{multline*}
It is easy to see that $\mathcal A^{[s,t]}_{1j}+\mathcal A^{[s,t]}_{2j}=1$, $j=1,2$. Therefore it remains to check that
$\mathcal A^{[s,t]}_{ij}\geq 0$.  We assume $a(t)+b(t)>1$, $\forall t\in (0,1)$ (the case $a(t)+b(t)<1$ can be considered similarly).
Then, since $0<a(t),b(t), 1-a(t),1-b(t)<1$ for all $t\geq 0$, we have
\begin{itemize}
\item[-] the  inequality $\mathcal A^{[s,t]}_{11}\geq 0$ is equivalent to $a(s)b(t)-\big(1-b(s)\big)\big(1-a(t)\big)\geq 0$
which is true since by our assumption we have $a(s)> 1-b(s)$ and $b(t)>1-a(t)$.
\item[-] the inequality $\mathcal A^{[s,t]}_{21}\geq 0$ is equivalent to $\frac{b(t)}{1-a(t)}\geq \frac{b(s)}{1-a(s)}$,
for all $s<t$. The last inequality follows from our condition that $a(t)>a(s)$ and $b(t)>b(s)$, for all $t>s$ (increasing functions).
\item[-] the inequality $\mathcal A^{[s,t]}_{12}\geq 0$ is equivalent to $\frac{a(t)}{1-b(t)}\geq \frac{a(s)}{1-b(s)}$,
for all $s<t$. The last inequality follows again from the condition that $a$ and $b$ are increasing functions.
\item[-] the  inequality $\mathcal A^{[s,t]}_{22}\geq 0$ is equivalent to $a(t)b(s)-\big(1-a(s)\big)\big(1-b(t)\big)\geq 0$
which is true since by our assumption we have $a(t)> 1-b(t)$ and $b(s)>1-a(s)$.
\end{itemize}
 This completes the proof.
\end{proof}
Denote
\[\alpha(s)=\beta_{111}^{(s)},\qquad \beta(s)=\beta_{112}^{(s)}, \qquad \gamma(s)=\beta_{211}^{(s)}, \qquad \delta(s)=\beta_{212}^{(s)}.\]
Using Theorem~\ref{tA} and Proposition~\ref{p3} we construct the
following cubic matrix
\begin{multline*}
\mathcal N^{[s,t]}= \left(P_{ijk}^{[s,t]}\right)=\frac{1}{a(t)+b(t)-1} \; \cdot\\
\left(\begin{array}{cc}
\alpha(s)b(t)+\beta(s)(a(t)-1)& \alpha(s)(b(t)-1)+\beta(s)a(t)\\[2mm]
(a(s)-\alpha(s))b(t)+(1-b(s)-\beta(s))(a(t)-1)& (a(s)-\alpha(s))(b(t)-1)+(1-b(s)-\beta(s))a(t)\\[3mm]
\end{array}
\right.\\
\vline \left.\begin{array}{cc}
\gamma(s)b(t)+\delta(s)(a(t)-1)& \gamma(s)(b(t)-1)+\delta(s)a(t)\\[2mm]
(1-a(s)-\gamma(s))b(t)+(b(s)-\delta(s))(a(t)-1)& (1-a(s)-\gamma(s))(b(t)-1)+(b(s)-\delta(s))a(t)
\end{array}
\right).
\end{multline*}

The following proposition illustrates Theorem~\ref{tA}.

\begin{pro} Let $a(t), b(t)\in (0,1)$ be functions such that $a(t)+b(t)-1>0$ for any $t\geq 0$.
The family of matrices $\mathcal N^{[s,t]}$ is a QSP of type $(12|a_0)$
if and only if the functions $\alpha(s)$, $\beta(s)$, $\gamma(s)$ and $\delta(s)$ satisfy the following
\begin{align}
0 &\leq \alpha(s)\leq a(s), \qquad  \quad   \ \; 0 \leq \beta(s)\leq 1-b(s), \ \  \ \forall s\geq 0; \label{albe} \\
0& \leq \gamma(s)\leq 1-a(s), \qquad     0 \leq \delta(s)\leq b(s), \qquad \quad  \forall s\geq 0. \label{gd}
\end{align}
\end{pro}
\begin{proof}
It is easy to see that the matrix $\mathcal N^{[s,t]}$ satisfies $\sum_{i,j}P_{ijk}^{[s,t]}=1$, for $k=1,2$.
Therefore we shall show that $P_{ijk}^{[s,t]}\geq 0$.
The system of inequalities $P_{1ij}^{[s,t]}\geq 0$, $i,j=1,2$,
is equivalent to the following system of inequalities with respect to $\alpha(s)$ and $\beta(s)$:
\begin{align*}
0 & \leq b(t)\alpha(s)-(1-a(t))\beta(s)\leq a(s)b(t)-(1-a(t))(1-b(s)),\\
 0 & \leq -(1-b(t))\alpha(s)+a(t)\beta(s)\leq -a(s)(1-b(t))+a(t)(1-b(s)).
\end{align*}

By dividing the first inequalities by $1-a(t)>0$ and by dividing the second inequalities by $a(t)>0$ and
by summing the resulting inequalities, we get
\[ 0\leq \frac{a(t)+b(t)-1}{(1-a(t))a(t)}\alpha(s)\leq a(s)   \frac{a(t)+b(t)-1}{(1-a(t))a(t)}.\]
Since $\frac{a(t)+b(t)-1}{(1-a(t))a(t)}>0$, we get the first inequalities of \eqref{albe}.
 The second inequalities of \eqref{albe} can be obtained similarly.

Now the system of inequalities $P_{2ij}^{[s,t]}\geq 0$, $i,j=1,2$,
is equivalent to the following system of inequalities with respect to $\gamma(s)$ and $\delta(s)$:
\begin{align*}
0 & \leq \ b(t)\gamma(s)-(1-a(t))\delta(s) \ \leq (1-a(s))b(t)-(1-a(t))b(s),\\
 0 & \leq -(1-b(t))\gamma(s)+a(t)\delta(s)\leq -(1-a(s))(1-b(t))+a(t)b(s).
\end{align*}
By dividing the first (resp. second) inequalities by $1-a(t)>0$ (resp. $a(t)>0$) and
by summing the resulting inequalities, we get
\[ 0\leq \frac{a(t)+b(t)-1}{(1-a(t))a(t)}\gamma(s)\leq (1-a(s))   \frac{a(t)+b(t)-1}{(1-a(t))a(t)}.\]
Again since $\frac{a(t)+b(t)-1}{(1-a(t))a(t)}>0$, we get the first inequalities of \eqref{gd}.
The second inequalities of \eqref{gd} can be obtained similarly.
\end{proof}

\section{An application to a population with possibility of twin birth}\label{bio}

For better biological interpretation, we renumber the set $I$, starting at 0 instead of 1. Consider the set $I=\{0,1,2\}$ as the set of types in a population.
The element 0 will play the role of an ``empty body'', the element 1 represents a ``female'',
while the element 2 is a ``male''.

Consider $P_{ijk}^{[s,t]}$ as the conditional
probability $P^{[s,t]}(i,j|k)$ that a member of
type $k\in I$ starts its ``pregnancy'' period at
time $s$ and finishes at time $t$ with zero (in case $i=j=0$),
one (in case $i=0$ or $j=0$, $i+j\ne 0$)
or two (in case $i\ne 0$ and $j\ne 0$, i.e. with a twin)
offspring of $i$th and $j$th types.

Then it is natural to define $P^{[s,t]}_{ijk}$ as follows:
\begin{align*}
P^{[s,t]}_{ij0}&=\begin{cases}
1, & \ \ \text{if} \ \ i=j=0,\\[2mm]
0, & \ \ \text{otherwise};
\end{cases} \\
P^{[s,t]}_{ij1} & \geq 0, \ \ \text{for all} \ \ i,j\in I,
\end{align*}
here $P^{[s,t]}_{001}$ can be strictly positive, which corresponds,
for example, to the case in which the female 1 cannot have a
child, because ``she'' is ill.
\[
P^{[s,t]}_{ij2}=\begin{cases}
1,& \ \ \text{if} \ \ i=j=0,\\[2mm]
0, &\ \ \text{otherwise}.
\end{cases}
\]
Hence the cubic matrix $\M^{[s,t]}=\left(P^{[s,t]}_{ijk} \right)_{i,j,k=0}^{\ \underset{2}{}}$ has the following form:
\begin{equation}\label{000}
\M^{[s,t]}= \begin{pmatrix}
1 &a^{[s,t]}& 1 &\vline &0 &\alpha^{[s,t]}&0& \vline &0&u^{[s,t]}&0\\
0 &b^{[s,t]}& 0 &\vline &0 &\beta^{[s,t]} &0& \vline &0&v^{[s,t]}&0\\
0 &c^{[s,t]}& 0 &\vline &0 &\gamma^{[s,t]}&0& \vline &0&w^{[s,t]}&0
\end{pmatrix}.
\end{equation}
It is easy to check that the functions $P^{[s,t]}_{ijk}$, for $k=0$ and $k=2$, satisfy
equation \eqref{eD} and condition \eqref{e12}, where one should use sums for $i,j=0,1,2$.
The equation \eqref{eD} for  $P^{[s,t]}_{ij1}$, taking into account the matrix \eqref{000},
can be written as the following system of nine equations
\begin{align}\label{ast}
a^{[s,t]}&=a^{[\tau,t]}+b^{[\tau,t]}+c^{[\tau,t]}+a^{[s,\tau]}\left(\alpha^{[\tau,t]}+\beta^{[\tau,t]}+\gamma^{[\tau,t]}\right)+
u^{[\tau,t]}+v^{[\tau,t]}+w^{[\tau,t]}, \\
\label{bst}
b^{[s,t]}&=b^{[s,\tau]}\left(\alpha^{[\tau,t]}+\beta^{[\tau,t]}+\gamma^{[\tau,t]}\right), \qquad c^{[s,t]}=c^{[s,\tau]}\left(\alpha^{[\tau,t]}+\beta^{[\tau,t]}+\gamma^{[\tau,t]}\right); \\
\alpha^{[s,t]}&=\alpha^{[s,\tau]}\left(\alpha^{[\tau,t]}+\beta^{[\tau,t]}+\gamma^{[\tau,t]}\right), \qquad \beta^{[s,t]}=\beta^{[s,\tau]}\left(\alpha^{[\tau,t]}+\beta^{[\tau,t]}+\gamma^{[\tau,t]}\right), \notag \\
\label{alphast}
\gamma^{[s,t]}&=\gamma^{[s,\tau]}\left(\alpha^{[\tau,t]}+\beta^{[\tau,t]}+\gamma^{[\tau,t]}\right); \\
u^{[s,t]}&=u^{[s,\tau]}\left(\alpha^{[\tau,t]}+\beta^{[\tau,t]}+\gamma^{[\tau,t]}\right), \qquad v^{[s,t]}=v^{[s,\tau]}\left(\alpha^{[\tau,t]}+\beta^{[\tau,t]}+\gamma^{[\tau,t]}\right), \notag \\
\label{ust}
w^{[s,t]}&=w^{[s,\tau]}\left(\alpha^{[\tau,t]}+\beta^{[\tau,t]}+\gamma^{[\tau,t]}\right).
\end{align}
Now we shall solve this system of two-variable-functional equations. By condition \eqref{e12} we should only consider
non-negative solutions which for any $0\leq s<t$ satisfy
\begin{equation}\label{bir}
a^{[s,t]}+b^{[s,t]}+c^{[s,t]}+\alpha^{[s,t]}+\beta^{[s,t]}+\gamma^{[s,t]}+u^{[s,t]}+v^{[s,t]}+w^{[s,t]}=1.
\end{equation}
Denoting
\[f(s,t)= \alpha^{[s,t]}+\beta^{[s,t]}+\gamma^{[s,t]}\]
from system \eqref{alphast} we get
 \[f(s,t)=f(s,\tau)f(\tau,t).\]
This equation is
known as Cantor's second equation which has a
very rich family of solutions:
\begin{itemize}
  \item[(a)] $f(s,t)\equiv 0$;
  \item[(b)] $f(s,t)=\frac{\Phi(t)}{\Phi(s)}$, where $\Phi$ is an arbitrary
function with $\Phi(s)\ne 0$;
   \item[(c)] \[f(s,t)=\begin{cases}
1, & \text{if} \  \ s\leq t<a,\\[2mm]
0, &  \text{if} \ \ t\geq a.\\
\end{cases} \quad \text{where} \ \ a>0.\]
\end{itemize}

\textsl{Case of the solution (a)}: In this case we get from system \eqref{ast}--\eqref{ust} and \eqref{bir} that
\[ a^{[s,t]}\equiv 1, \qquad  b^{[s,t]}=c^{[s,t]}=\alpha^{[s,t]}=\beta^{[s,t]}=\gamma^{[s,t]}=u^{[s,t]}=v^{[s,t]}=w^{[s,t]}\equiv 0.\]
Thus we constructed a QSP of type $(12|a_0)$. To give a biological interpretation, let us compute distributions
$x^{(t)}=(x^{(t)}_0, x^{(t)}_1, x^{(t)}_2)$. By using formula \eqref{xt12} we get
\begin{align*}
x^{(t)}_0& =\frac{1}{2} \sum_{i,j=0}^2\left(P_{0ij}^{[s,t]}+P_{i0j}^{[s,t]}\right)x^{(s)}_j= x^{(s)}_0+x^{(s)}_1+x^{(s)}_2=1,\\
 x^{(t)}_1 &=x^{(t)}_2=0,  \qquad \quad 0\leq s<t.
\end{align*}
 Since $x^{(s)}_0=1$ is the probability to have $0$ type, the biological
interpretation of the process is clear: independently on initial distribution $x^{(s)}$, the population will die
as soon as the time $t>s$ turns on.

\textsl{Case of the solution (b)}: In this case from system \eqref{bst} we get
\[b^{[s,t]}=b^{[s,\tau]}\left(\alpha^{[\tau,t]}+\beta^{[\tau,t]}+\gamma^{[\tau,t]}\right)
=b^{[s,\tau]}f(\tau,t)=b^{[s,\tau]}\frac{\Phi(t)}{\Phi(\tau)},\] consequently,
\[\frac{b^{[s,t]}}{\Phi(t)}=\frac{b^{[s,\tau]}}{\Phi(\tau)}.\]
From the last equality it follows that the function $\frac{b^{[s,t]}}{\Phi(t)}$ should not depend on $t$, i.e. there exists
a function, say $b(s)$, such that
\[\frac{b^{[s,t]}}{\Phi(t)}=b(s) \ \ \ \Rightarrow \ \ \ b^{[s,t]}=b(s)\Phi(t).\]
Similarly one can prove that there are functions $c(s), \alpha(s), \beta(s), \dots, w(s)$ such that
\begin{align}\label{aab}
c^{[s,t]}&=c(s)\Phi(t), \quad \alpha^{[s,t]}=\alpha(s)\Phi(t), \quad \beta^{[s,t]}=\beta(s)\Phi(t), \quad
\gamma^{[s,t]}=\gamma(s)\Phi(t),\notag{} \\
u^{[s,t]}&=u(s)\Phi(t), \quad v^{[s,t]}=v(s)\Phi(t), \quad w^{[s,t]}=w(s)\Phi(t).
\end{align}
By definition of $f(s,t)$ we shall have
\[f(s,t)=\alpha^{[s,t]}+\beta^{[s,t]}+\gamma^{[s,t]}=\Phi(t)(\alpha(s)+\beta(s)+\gamma(s))=\frac{\Phi(t)}{\Phi(s)},\]
i.e.
\begin{equation}\label{psi}
\alpha(s)+\beta(s)+\gamma(s)=\frac{1}{\Phi(s)}.
\end{equation}
By using equalities \eqref{aab} and \eqref{psi} from \eqref{ast} we get
\[a^{[s,t]}=a^{[\tau,t]}+a^{[s,\tau]}\frac{\Phi(t)}{\Phi(\tau)}+\Phi(t)\big(b(\tau)+c(\tau)+u(\tau)+v(\tau)+w(\tau)\big).\]
Denoting $g(s,t)=\frac{a^{[s,t]}}{\Phi(t)}$ from the last equation we get
\[g(s,t)=g(s,\tau)+g(\tau,t)+b(\tau)+c(\tau)+u(\tau)+v(\tau)+w(\tau).\]
This equation has the following solution\footnote{The equation $g(s,t)=g(s,\tau)+g(\tau,t)$ is known as Cantor's first equation.
It is easy to check that this equation has very rich class of solutions, i.e. $g(s,t)=\kappa(t)-\kappa(s)$ is a solution for an \textsl{arbitrary} function $\kappa$. For Cantor's first and second equations, see http://eqworld.ipmnet.ru/en/solutions/eqindex/eqindex-fe.htm.}
\[g(s,t)=\kappa(t)-\kappa(s)-b(s)-c(s)-u(s)-v(s)-w(s), \]
where $\kappa(t)$ is an arbitrary function. Consequently, we get
\[a^{[s,t]}=\Phi(t)\Big(\kappa(t)-\kappa(s)-b(s)-c(s)-u(s)-v(s)-w(s)\Big).\]
Thus we obtained a solution of the system \eqref{ast}--\eqref{ust}, for which the
condition \eqref{bir} has the form
\[
\Phi(t)\Big(\kappa(t)-\kappa(s)+\frac{1}{\Phi(s)}\Big)=1,
\] i.e.
\[\kappa(s)-\frac{1}{\Phi(s)}=\kappa(t)-\frac{1}{\Phi(t)}.\]
This equality says that the function  $\kappa(t)-\frac{1}{\Phi(t)}$
should not depend on $t$, i.e. there is a constant $K$ such that
\[\kappa(t)=K+\frac{1}{\Phi(t)}.\]
Thus
\[a^{[s,t]}=\Phi(t)\Big(\frac{1}{\Phi(t)}-\frac{1}{\Phi(s)}-b(s)-c(s)-u(s)-v(s)-w(s)\Big).\]
Now we are ready to write an explicit formula for the corresponding cubic matrix:

\begin{multline} \label{0m}
\M^{[s,t]}=\Phi(t) \left(\begin{array}{cccccccccccc}
1 &\frac{1}{\Phi(t)}-\frac{1}{\Phi(s)}-b(s)-c(s)-u(s)-v(s)-w(s)& 1 &\\
0 &b(s)& 0 &\\
0 &c(s)& 0 &
\end{array}
\right. \\
\vline \left.
\begin{array}{cccccccccccc}
&0 &\alpha(s)&0& \vline &0&u(s)&0\\
&0 &\beta(s) &0& \vline &0&v(s)&0\\
&0 &\frac{1}{\Phi(s)}-\alpha(s)-\beta(s)&0& \vline &0&w(s)&0
\end{array}
\right).
\end{multline}
Thus we have proved the following.
 \begin{pro} The family of matrices $\mathcal M^{[s,t]}$, \eqref{0m}, is a QSP of type $(12|a_0)$ if and only if $b(t), c(t),\alpha(t),\beta(t),u(t),v(t),w(t) \in [0,1]$, $\Psi(t)>0$, are arbitrary functions
 such that
 \begin{align*}
\frac{1}{\Phi(t)}-\frac{1}{\Phi(s)} & \geq b(s)+c(s)+u(s)+v(s)+w(s),\\
\alpha(t)+\beta(t) & \leq \frac{1}{\Phi(t)}, \quad \text{for all} \ \ 0\leq s<t.
 \end{align*}
 \end{pro}
Time behavior of this QSP depends on fixed functions.
Let us give some interesting interpretations:
\begin{itemize}
\item[-]
Assume the following limit exists
\[\lim_{t\to\infty}\Phi(t)=\Phi(\infty)>0.\]
In this case, if, for example, $\beta(s)>0$ then
the population has a positive
probability, $P_{111}^{[s,\infty]}=\Phi(\infty)\beta(s)>0$ to have twins (i.e. female-female twins).
In this case, the condition $\gamma(s)+v(s)>0$ gives positive probability  $\Phi(\infty)(\gamma(s)+v(s))$
of having female-male twins, and similarly if $w(s)>0$ then $\Phi(\infty)w(s)>0$ is the positive probability
of male-male twins.

\item[-] If at the initial time $s$ some probability is positive, for example, $\Psi(t)\beta(s)>0$, then
during all fixed time $t$, with $t>s$, this probability remains positive.
For this example, it means that if initially, the population had possibility,
to have a female-female twin, then it will have this possibility always, although this might not be true in the limiting case.
\item[-] If $\Phi(\infty)=0$ then the population asymptotically dies.
\item[-] It is known that the human twin birth rate is about $2$ percent of standard one-child birth. This can be used
to choose our parameter functions. For example, one can take $\beta(t)=0.02\alpha(t)$, etc.
\end{itemize}
\textsl{Case of the solution (c)}: In this case for $t<a$ we have
\[ \alpha^{[\tau,t]}+\beta^{[\tau,t]}+\gamma^{[\tau,t]}=1, \qquad  0\leq \tau<t<a.\]
Then from the system \eqref{bst}--\eqref{ust} we get
 that there are functions $b_0(s), c_0(s), \dots, w_0(s)$ such that
\begin{align*}
b^{[s,t]}&=b_0(s), \quad c^{[s,t]}=c_0(s), \quad \, \alpha^{[s,t]}=\alpha_0(s), \quad \beta^{[s,t]}=\beta_0(s),\quad
\gamma^{[s,t]}=\gamma_0(s),\\
u^{[s,t]}&=u_0(s), \quad  v^{[s,t]}=v_0(s), \quad w^{[s,t]}=w_0(s), \qquad \quad  0\leq \tau<t<a.
\end{align*}
Then by equation \eqref{ast} we get that there is $\kappa_0(t)$ such that
\[a^{[s,t]}=\kappa_{0}(t)-\kappa_0(s)-b_0(s)-c_0(s)-u_0(s)-v_0(s)-w_0(s).\]
Then by condition \eqref{bir} we get that $\kappa_0(t)=\kappa_0(s)$. To make the corresponding
matrix a (1,2)-stochastic we need $b_0(s),c_0(s),u_0(s),v_0(s),w_0(s)\in [0,1]$ and by previous results
we get
\[a^{[s,t]}=-b_0(s)-c_0(s)-u_0(s)-v_0(s)-w_0(s),\]
which is non-negative if and only if
\[b_0(s)=c_0(s)=u_0(s)=v_0(s)=w_0(s)\equiv 0.\]
Thus, for $t<a$, the cubic matrix has the following form
\[
\M_1^{[s,t]}=\begin{pmatrix}
1 &0& 1 &\vline&0 &\alpha_0(s)&0& \vline &0&0&0\\
0 &0& 0 &\vline&0 &\beta_0(s) &0& \vline &0&0&0\\
0 &0& 0 &\vline&0 &1-\alpha_0(s)-\beta_0(s)&0& \vline &0&0&0
\end{pmatrix}.
\]
The case $t\geq a$ is simpler, because in this case $f(s,t)=0$ and the solution of the system
is
\[a^{[s,t]}=1, \qquad  b^{[s,t]}=\dots=w^{[s,t]}=0.\]
Thus, for $t\geq a$, the cubic matrix has the following form
\[
\M_0^{[s,t]}=
\begin{pmatrix}
1 &1& 1 & \vline & 0 &0&0& \vline &0&0&0 \\
0 &0& 0 & \vline &0 &0&0 & \vline &0&0&0 \\
0 &0& 0 & \vline &0 &0&0 & \vline &0&0&0 \\
\end{pmatrix}.
\]
Define now
\begin{equation}\label{mc}
\M^{[s,t]}=\begin{cases}
\M_1^{[s,t]}, & \ \ \text{if}\ \ s\leq t<a,\\[2mm]
\M_0^{[s,t]}, & \ \ \text{if} \ \ t\geq a,\\
\end{cases}\ \ \text{where} \ \ a>0.
\end{equation}
Thus we have proved the following.
 \begin{pro} The family of matrices $\mathcal M^{[s,t]}$, \eqref{mc}, is a QSP of type $(12|a_0)$
 if and only if $\alpha_0(t),\beta_0(t)\in [0,1]$, are arbitrary functions
 such that
 \[\alpha_0(t)+\beta_0(t)\leq 1, \quad \text{for all} \ \ 0\leq s<t<a.\]
 \end{pro}
 The time behavior of this QSP depends on fixed functions.
But it is  simpler than the previous case. Let us give some interesting interpretations:
\begin{itemize}
\item[-] Start the process at time $s$, then for any $t<a$, the
probabilities $\{P_{101}^{[s,t]},P_{111}^{[s,t]}, P_{121}^{[s,t]}\} =\{\alpha_0(s),\beta_0(s), 1-\alpha_0(s)-\beta_0(s)\}$ are
independent on time $t<a$. While all the other probabilities independent on both $s$ and $t$. Thus the process is stable
for any $t<a$, i.e. until $t=a$.

\item[-] Start the process at time $s$ as soon as $t\geq a$ then the population immediately dies.
This phenomenon reminds a cataclysm (catastrophe): ``everything is going good, good, \dots, died''.
\end{itemize}
\section*{Acknowledgements}

 This work was partially supported by  Agencia Estatal de Investigaci\'on (Spain),
grant MTM2016-79661-P (European FEDER support included, UE)  and by Kazakhstan Ministry of Education and Science, grant 0828/GF4: ``Algebras, close to Lie: cohomologies, identities and deformations''.

\end{document}